\documentclass[a4paper,10pt]{article}

\usepackage[T1]{fontenc}
\usepackage[utf8]{inputenc}
\usepackage{amsmath,amssymb,amsthm}
\usepackage{enumerate}
\usepackage{geometry}
\usepackage[all]{xy}
\xyoption{arc}

\newtheorem{theorem}{Theorem}
\newtheorem{corollary}{Corollary}
\newtheorem{remark}{Remark}

\DeclareMathOperator{\diag}{diag}

\begin{document}

\title{A Constructive Cayley Representation of Orthogonal Matrices\\
and Applications to Optimization}

\author{Iwo Biborski}
\date{}

\maketitle

\footnotetext{Author's e-mail: iwo.biborski@gmail.com}

% ---------- Abstract ----------
\begin{abstract}
It is known that every real orthogonal matrix can be brought into the domain
of the Cayley transform by multiplication with a suitable diagonal signature matrix.
In this paper we provide a constructive and numerically efficient algorithm
that, given a real orthogonal matrix $U$, computes a diagonal matrix
$D$ with entries in $\{\pm1\}$ such that the Cayley transform of $DU$
is well defined.
This yields a representation of $U$ in the form
\[
U = D(I-S)(I+S)^{-1},
\]
where $S$ is a skew-symmetric matrix.
The proposed algorithm requires $O(n^{3})$ arithmetic operations and
produces an explicit quantitative bound on the associated skew-symmetric
generator.
As an application, we show how this construction can be used to control
singularities in Cayley-transform-based optimization methods on the
orthogonal group.
\end{abstract}

% ---------- Introduction ----------
\section{Introduction}

Let \(X\) be an \(n\times n\) real matrix and let \(X^{T}\) denote its transpose.
By \(I\) we denote the identity matrix of appropriate size.
A matrix \(U\) is called \emph{orthogonal} if \(UU^{T}=U^{T}U=I\).

For a matrix \(A\) such that \(\det(A+I)\neq 0\), the \emph{Cayley transform} is defined by
\[
C(A)=(I-A)(A+I)^{-1}=-I+2(A+I)^{-1}.
\]
When restricted to orthogonal matrices without eigenvalue \(-1\),
the Cayley transform is a bijection onto the space of skew-symmetric matrices.

It was shown in \cite{OsLie,Ka,Bib,Od} that for any real matrix \(A\)
(or more generally over a field of characteristic different from \(2\)),
there exists a diagonal matrix \(D\) with diagonal entries in \(\{-1,1\}\)
such that \(\det(A+D)\neq 0\).
Consequently, for any orthogonal matrix \(U\), there exists a diagonal
\emph{signature matrix} \(D\) such that the Cayley transform \(C(DU)\) is well defined.

The known proofs of existence of such a matrix \(D\) are nonconstructive.
The main purpose of this paper is to provide a \emph{constructive and efficient}
algorithm for computing \(D\).

\vspace{1ex}

We recall below an earlier result that motivates the present work. We put $k_a(A,\lambda)$ for an algebraic multiplicity of eigenvalue $\lambda$ of a matrix $A$. We have the following theorem

\begin{theorem}[{\cite[Theorem~1]{Bib}}]\label{thm:old}
Let \(A\) be an \(n\times n\) matrix over a field of real or complex numbers
and let \(\lambda\neq 0\) be an eigenvalue of \(A\).
Then there exists an index \(i\in\{1,\dots,n\}\) such that
\[
k_a(D_iA,\lambda) < k_a(A,\lambda),
\]
where \(D_i\) is the diagonal matrix whose \(i\)-th diagonal entry equals \(-1\)
and all others equal \(1\).
\end{theorem}

This result yields a constructive but non-optimal procedure for computing a
signature matrix by successive reduction of eigenvalue multiplicity.

% ---------- Main result ----------
\section{Main Result}

\begin{theorem}\label{thm:main}
Let $A$ be an $n\times n$ matrix in a field of real or complex numbers.
Then there exists a diagonal matrix
$D=\diag(d_1,\dots,d_n)$ with $d_i\in\{\pm1\}$ such that
\[
|\det(A+D)| \ge 1.
\]
Consequently, $-1$ is not an eigenvalue of $DA$, and the Cayley transform
$C(DA)$ is well defined.
Moreover, such a matrix $D$ can be computed using a Gaussian-elimination-type
algorithm using arithmetic operations $O(n^3)$.
\end{theorem}

\begin{proof}
Since $D^{2}=I$, we have
\[
|\det(DA+I)|=|\det(DA+D^{2})|=|\det(D(A+D))|=|\det(A+D)|,
\]
and therefore
\begin{equation}\tag{4}
|\det(DA+I)|\geq 1 \quad\Longleftrightarrow\quad |\det(A+D)|\geq 1.
\end{equation}
Hence, it suffices to show that there exists a diagonal matrix
$D=\diag(d_1,\dots,d_n)$ with $d_i\in\{\pm1\}$ such that $|\det(A+D)|\geq 1$.
We proceed by induction on $n$.

\medskip
\noindent\emph{Base case ($n=2$).}
Let
\[
A+D=
\begin{bmatrix}
a_{11}+d_1 & a_{12} \\
a_{21} & a_{22}+d_2
\end{bmatrix}.
\]
Since the field has characteristic different from $2$, either
$a_{11}+1\neq 0$ or $a_{11}-1\neq 0$, we may choose $d_1$ such that
$|a_{11}+d_1|\ge 1$.
We then apply a single step of Gaussian elimination to eliminate the entry
$a_{21}$, obtaining
\[
\begin{bmatrix}
a_{11}+d_1 & a_{12} \\
0 & a_{22}-\dfrac{a_{12}a_{21}}{a_{11}+d_1}+d_2
\end{bmatrix}.
\]
Again using that the field has a characteristic different from $2$, we may choose
$d_2\in\{\pm1\}$ so that the second diagonal entry absolute value is greater or equal to 1.
Thus $|\det(A+D)|\geq 1$.
\medskip

\noindent\emph{Induction step.}
Assume the claim holds for matrices of size $(n-1)\times(n-1)$.
Let $A$ be an $n\times n$ matrix and consider $A+D$ with
$D=\diag(d_1,\dots,d_n)$.
Choose $d_1\in\{\pm1\}$ so that $|a_{11}+d_1|\ge 1$.
Applying Gaussian elimination using only row replacement operations
$R_i\leftarrow R_i+cR_1$, which preserve the determinant, we eliminate the
entries below $a_{11}+d_1$ and obtain a matrix of the form
\[
\begin{bmatrix}
a_{11}+d_1 & * \\
0 & A_1'+D_1
\end{bmatrix},
\]
where $A_1'$ is an $(n-1)\times(n-1)$ matrix and
$D_1=\diag(d_2,\dots,d_n)$.

By the induction hypothesis, there exist signs $d_2,\dots,d_n$ such that
$|\det(A_1'+D_1)|\geq 1$. Consequently,
\[
|\det(A+D)|=|(a_{11}+d_1)|\cdot|\det(A_1'+D_1)|\geq 1.
\]

This completes the induction and the proof.
\end{proof}
The proof above yields an explicit algorithm for computing a suitable
signature matrix. The procedure proceeds iteratively by selecting the
diagonal entries $d_i$ and applying Gaussian elimination to eliminate
subdiagonal entries in the $i$-th column.

\begin{corollary}
Every orthogonal matrix \(U\) admits a representation
\[
U = D(I-S)(I+S)^{-1},
\]
where \(D\) is a diagonal matrix with entries in \(\{-1,1\}\)
and \(S\) is skew-symmetric.
Such a representation can be computed in \(O(n^{3})\) arithmetic operations.
\end{corollary}

\begin{remark}[Compression of orthogonal matrices]
The proposed factorization shows that any real orthogonal matrix can be
losslessly encoded by a bounded skew-symmetric matrix and a diagonal sign
matrix.
This reduces the storage of an orthogonal matrix from $n^2$ real parameters to
$\frac{n(n-1)}{2}$ real parameters and $n$ bits.
Such a representation may be useful in applications where memory usage or data
transfer is critical, while exact reconstruction of the orthogonal matrix is
required.
\end{remark}

% ---------- Application ----------
\section{Application to Optimization on \(SO(n)\)}

The Cayley transform provides a natural parametrization for optimization
problems on the orthogonal group.
However, standard Cayley-based methods suffer from singularities near matrices
with eigenvalue \(-1\).

The following result shows that the pivoting strategy described above maps the
entire orthogonal group into a fixed bounded subset of the space of
skew-symmetric matrices.

\begin{theorem}\label{thm:bound}[Uniform boundedness of the pivoted Cayley transform]
Let $U \in O(n)$ be an arbitrary orthogonal matrix.
There exists a signature matrix $D$ such that
\[
|\det(DU+I)| \geq 1.
\]
Then the Cayley transform
\[
C(DU) := (I-DU)(I+DU)^{-1}
\]
is well defined and belongs to a fixed bounded subset of the space of
skew-symmetric matrices. More precisely,
\[
\rho\big(C(DU)\big) \le 1 + 2^n,
\]
where $\rho(\cdot)$ denotes the spectral radius.
In particular, the pivoted Cayley transform maps the entire orthogonal group
$O(n)$ into a bounded subset of $\mathfrak{so}(n)$.
\end{theorem}

\begin{proof}
Applying the algorithm of Theorem~\ref{thm:main} to the matrix $U$,
we obtain a signature matrix $D$ such that
\[
|\det(DU+I)|\ge 1.
\]
Since $DU$ is orthogonal, all eigenvalues of $DU$ lie on the unit circle.
Consequently, every eigenvalue of $DU+I$ has absolute value at most $2$.

Let $\lambda_{\min}$ denote an eigenvalue of $DU+I$ of minimal absolute value.
Then
\[
|\det(DU+I)|
= \prod_{j=1}^n |\lambda_j|
\le 2^{\,n-1} |\lambda_{\min}|,
\]
which together with $|\det(DU+I)|\ge 1$ implies
\[
|\lambda_{\min}|\ge 2^{1-n}.
\]
Hence, all eigenvalues of $(DU+I)^{-1}$ have absolute value at most $2^{n-1}$,
and therefore
\[
\rho\big((DU+I)^{-1}\big)\le 2^{n-1}.
\]

Since
\[
C(DU)=-I+2(DU+I)^{-1},
\]
we obtain the estimate
\[
\rho\big(C(DU)\big)
\le \|I\|_2 + 2\|(DU+I)^{-1}\|_2
\le 1 + 2^n.
\]
In particular, all eigenvalues of $C(DU)$ are bounded in absolute value by
$1+2^n$, and thus $C(DU)$ belongs to a bounded subset of the space of
skew-symmetric matrices.
\end{proof}

\begin{remark}
The bound \(2^{n}\) is not optimal.
The purpose of Theorem~\ref{thm:bound} is to establish uniform boundedness
independent of \(U\), which is sufficient for applications to global
optimization on \(SO(n)\).
\end{remark}

\begin{remark}
O'Dorney \cite{Od} shows a stronger result: there exists a signature matrix $D$ such that all entries of the Cayley transform $C(DU)$ lie in the interval $[-1,1]$. However, the proof is non-constructive and does not yield an efficient algorithm. It remains an open problem whether one can compute such a matrix $D$ with $O(n^3)$ arithmetic operations.
\end{remark}

\begin{remark} [Quantitative control of Cayley re-centering]
The use of the Cayley transform in optimization on orthogonal groups is known
to suffer from singularities when the iterates approach matrices with
eigenvalue $-1$, leading to extremely slow convergence due to the blow-up of
the associated skew-symmetric parameter.
This phenomenon and a practical remedy based on changing the origin of the
Cayley transform, referred to as \emph{pivoting}, were analyzed in
\cite{HoriTanaka2010}.

The construction presented above provides a deterministic and quantitative
variant of such re-centering strategies.
Given an orthogonal matrix $U$, one can select a diagonal signature matrix
$D=\mathrm{diag}(\pm1)$ such that the Cayley transform
\[
C(DU)=(I-DU)(I+DU)^{-1}
\]
is well defined and satisfies the explicit bound
\[
\rho\big(C(DU)\big)\le 1+2^n.
\]
Thus, whenever the skew-symmetric representation associated with the current
iterate becomes large, one may replace $U$ by $DU$ and continue the
optimization within a bounded coordinate chart of the Lie algebra.

While the bound grows exponentially with the dimension, it already yields
effective control in small to moderate dimensions, which are typical in
applications of Cayley-transform-based optimization.
In contrast to heuristic pivoting rules, this approach provides an explicit
a priori bound on the size of the skew-symmetric parameter after re-centering.
\end{remark}

\section*{Conclusion}
We have presented a constructive algorithm that, given any real orthogonal
matrix $U$, computes a diagonal signature matrix $D$ such that the Cayley
transform $C(DU)$ is well defined.
This provides a representation
\[
U = D(I-S)(I+S)^{-1}, \qquad S \text{ skew-symmetric},
\]
and yields two immediate applications:
\begin{enumerate}
    \item In optimization on the orthogonal group, the method allows
    quantitative re-centering to avoid singularities near eigenvalue $-1$.
    \item Any orthogonal matrix can be compactly encoded via a bounded
    skew-symmetric matrix and a diagonal signature matrix, reducing storage
    requirements from $n^2$ reals to $\frac{n(n-1)}{2}$ reals plus $n$ bits.
\end{enumerate}
While the bound on the skew-symmetric parameter grows exponentially with
dimension, the algorithm is efficient and practical for small to moderate
$n$. 
An open question remains whether one can constructively achieve a uniform
entrywise bound, as in \cite{Od}, with comparable computational complexity.

% ---------- Bibliography ----------


\begin{thebibliography}{99}

\bibitem{OsLie}
H.~Liebeck and A.~Osborne,
\textit{The generation of all rational orthogonal matrices},
Amer. Math. Monthly 98 (1991), 131--133.

\bibitem{Ka}
W.~Kahan,
\textit{Is there a small skew Cayley transform with zero diagonal?}
Linear Algebra Appl. 417 (2006), 335--341.

\bibitem{Bib}
I.~Biborski,
\textit{On computation of skew-symmetric generator for an orthogonal matrix},
UIAM (2012), 47--51.

\bibitem{Od}
E.~O'Dorney,
\textit{Minimizing the Cayley transform of an orthogonal matrix by multiplying
by signature matrices},
Linear Algebra Appl. 448 (2014), 97--103.

\bibitem{Gall}
J.~Gallier,
\textit{Remarks on the Cayley representation of orthogonal matrices},
arXiv:math/0606320.

\bibitem{HoriTanaka2010}
G.~Hori and T.~Tanaka,
\textit{Pivoting in Cayley transform-based optimization on orthogonal groups},
Proc. APSIPA Annual Summit and Conf., 2010.


\end{thebibliography}
\end{document}